\documentclass[11pt]{article}
\usepackage[T1]{fontenc}
\usepackage{textcomp}
\usepackage{longtable}

\usepackage{tikz}
\usepackage{verbatim}
\usetikzlibrary{arrows}

\usepackage[scaled=0.92]{helvet}
\usepackage{times}
\usepackage{amsfonts}
\usepackage{amsmath}
\usepackage{amssymb}
\usepackage{graphics}

\usepackage{theorem}
\usepackage{enumerate}

\newcommand{\thrmrefer}[1]{\renewcommand\thethrm{\protect\ref{#1}}\addtocounter{thrm}{-1}}

\def\bg{{\overline{\Gamma}}}
\def\rank{{\rm rank}}

\def\End{\rm {End}}

\def\tk1{{\mbox{\scriptsize $K_1$}}}

\def \GF{\mathop{\mathrm{GF}}}

\newcommand{\Gr}{\mathop{\mathrm{Gr}}}
\newcommand{\PSL}{\mathop{\mathrm{PSL}}}
\newcommand{\PGamL}{\mathop{\mathrm{P}\Gamma\mathrm{L}}}
\newcommand{\PSp}{\mathop{\mathrm{PSp}}}

\newcommand{\bgam}{\overline{\Gamma}}
\newcommand{\gam}{\Gamma}

\newcommand{\cp}{\begin{proof} }
\newcommand{\tp}{\end{proof} }
\newtheorem{lemma}{Lemma}[section]

\newtheorem{theorem}[lemma]{Theorem}

\newtheorem{prob}[lemma]{Problem}

  \newenvironment{proof}{\noindent {\bf Proof:}\ }{~\rule{1.5ex}{1.5ex}\\}

  \newcounter{numpasso}

\title{Groups Synchronizing a  Transformation of Non-Uniform Kernel}
\author{Jo\~{a}o Ara\'{u}jo
\and
Wolfram Bentz
\and
Peter J. Cameron}
\date{\today}

\begin{document}

 \newenvironment{proofof}{\noindent {\sc Proof of Theorem}} {\phantom{a} \hfill \framebox[2.2mm]{ } \bigskip}
\maketitle

\begin{abstract}

This paper concerns the general problem of classifying the finite deterministic automata  that admit a synchronizing (or reset) word. (For our purposes it is irrelevant if the automata has initial  or final states.) Our departure point is the study of the transition semigroup associated to the automaton, taking advantage of the enormous and very deep progresses made during the last decades on the theory of permutation groups, their geometry and their combinatorial structure.  

Let $X$ be a finite set. We say that a primitive group $G$ on $X$ is {\em synchronizing} if $G$ together with any non-invertible map on $X$ generates a constant map. 
It is known (by some recent results proved by P. M. Neumann) that for some primitive groups $G$ and for some singular transformations $t$ of uniform kernel (that is, all blocks have the same number of elements), the semigroup $\langle G,t\rangle$ does not generate a constant map.   Therefore the following concept is very natural:  a primitive group $G$ on $X$ is said to be {\em almost synchronizing} if $G$ together with any map of non-uniform kernel generates a constant map.  In this paper we use two different methods to provide several infinite families of groups that are not synchronizing, but are almost synchronizing. The paper ends with a number of problems on synchronization likely  to attract the attention of experts in computer science, combinatorics and geometry, groups and semigroups, linear algebra and matrix theory.

 $\ $

\noindent{\em $2010$ Mathematics Subject Classification\/}: 68Q15, 68Q70, 20B15, 20B25, 20B40, 20M20, 20M35.
\end{abstract}

\section{Introduction}

Imagine that you are in a dungeon consisting of a number of interconnected caves, all of which appear identical. Each cave has a number of one-way doors of different colors through which you may leave; these lead to passages to other caves. There is one more door in each cave; in one cave the extra door leads to freedom, in all the others to instant death. You have a map of the dungeon with the escape door identified, but you do not know in which cave you are. If you are lucky, there is a sequence of doors through which you may pass which take you to the escape cave from any starting point. The example below shows this.

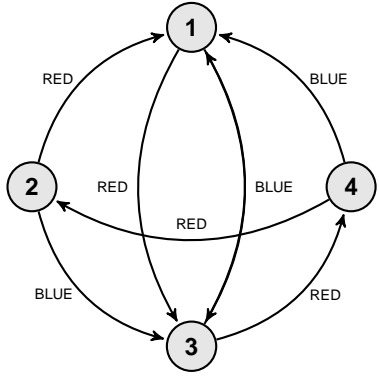
\begin{center}
\begin{tikzpicture}[->,>=stealth',shorten >=1pt,auto,node distance=3cm,
  thick,main node/.style={circle,fill=black!10,draw,font=\sffamily\small\bfseries}]

  \node[main node] (1) {1};
  \node[main node] (2) [below left of=1] {2};
  \node[main node] (3) [below right of=2] {3};
  \node[main node] (4) [below right of=1] {4};

  \path[every node/.style={font=\sffamily\tiny}]
    (1) 
	 edge [bend right] node[left] {RED} (3)
	 edge [bend left] node[right] {BLUE} (3)
    (2) edge [bend left] node [left] {RED} (1)
        edge [bend right] node[left] {BLUE} (3)
    (3) 
        edge [bend right] node[right] {RED} (4)
        edge [bend right] node[left] {} (1)
    (4) 
        edge [bend right] node[right] {BLUE} (1)
        edge [bend left] node[above] {RED} (2);
\end{tikzpicture}
\end{center}

In the example, the sequence (BLUE, RED,BLUE,BLUE) always brings you to cave number 3.  In this situation the dungeon (in fact a finite deterministic automaton) admits a {\em synchronizing}  word (or a {\em reset} word) and in general we are interested in the following questions: for a given finite deterministic automaton,
\begin{enumerate}
\item does there exist a synchronizing word?
\item if so, can we bound its length?
\end{enumerate}

For instructive  illustrations of the  practical usefulness of synchronization please see  \cite{AnVo,Tr11} and also Volkov's  talk \cite{Vo}; for the fast growing bibliography on the topic we refer the reader to the websites \cite{JEP,Tr}. 

The second of the questions above points to one of the oldest and most famous problems in automata theory, the so called \v Cern\'y problem, that prompted the publication of over one hundred papers, and the organization of several conferences and workshops fully dedicated to his conjetcure.  The first question is our concern in this paper, since before knowing how long it will take to escape the dungeon, the person wants to know if there is an escape at all. In the literature there exist algorithms to decide whether  a given automaton admits a synchronizing word or not  (see for example \cite{epp,Tr11}). But our approach is different.  To every deterministic automaton we can associate in a standard way  a semigroup and hence translate the questions above to the language of semigroups of transformations: for a given set $S$ of transformations  on a finite set,

\begin{enumerate}
\item  does $S$ generate a constant map?
\item if so, can we bound the length of a word, on the generators, that gives a constant map?  
\end{enumerate} 

One general goal is to try to answer these questions using the algebraic properties of the semigroup. In particular, if $S$  generates a group of permutations, it is possible to take advantage of the enormous body of knowledge and deep results that have been proved on them for the past decades. In fact we do not really need that the semigroup contains a group of permutations. For instance, if $K=g_1a_{1}g_2a_{2}\ldots g_ka_{k}g_{k+1}$ is a constant map (when the $a_{i}$ are singular maps and the $g_{i}$ are permutations), then 
\[
K=g_1a_{1}g_2a_{2}\ldots g_ka_{k}g_{k+1}=a^{g_1}_{1}a^{g_1g_2}_{2}a^{g_1g_2g_3}_{3}\ldots
a^{g_1g_2\ldots g_{k}}_{k} g_1g_2\ldots g_{k}g_{k+1},
\](where $a^g=gag^{-1}$).
 Thus $a^{g_1}_{1}a^{g_1g_2}_{2}a^{g_1g_2g_3}_{3}\ldots
a^{g_1g_2\ldots g_{k}}_{k}$ is also a constant, and is generated by conjugates of the $a_{i}$. 
Therefore, if we know that a given group $G$ together with any singular map generates a constant, then we know that any set of singular maps whose normalizer contains $G$, also generates a constant. (The normalizer of a set $S$ of transformations on $X$ is the group of permutations $g$ of $X$ such that $s^{g}\in \langle S\rangle$, for all $s\in S$.)

For a natural $n$, let $S_{n}$ and $T_{n}$ be, respectively, the symmetric group and the full transformation monoid on the set $X=\{1,\ldots,n\}$. 
We say that a group $G\leq S_{n}$ synchronizes a transformation $t\in T_{n}\setminus S_{n}$ if the semigroup $\langle G,t\rangle$ contains a constant map. 

In this paper we take the stand of  \cite{ben,pmn}, that is, we are interested in the groups that together with any  non-invertible transformation $t\in T_{n}\setminus S_{n}$ generate a constant.   A group $G\leq S_{n}$ is said to be non-synchronizing if there exists a partition $P$ of the set $X$ and a section $S$ such that, for all $g\in G$,  $Sg$ is a  section for $P$. Such a section is said to be $G$-regular and it is proved in \cite{pmn} that any partition admitting a $G$-regular section must be uniform (that is, all its blocks have the same size). A group that is not non-synchronizing is said to be synchronizing; a synchronizing group $G\leq S_{n}$ synchronizes every non-invertible transformation $t\in T_{n}$. (This result was first proved by the first author  \cite{ar} and was included in \cite{pmn}.)  

A synchronizing group must be primitive (\cite{ar,ben,pmn}).  In addition, if we have a singular map $t$ such that $\rank (t)>1$ and a primitive group $G$, that together with $t$, cannot generate maps of lower rank than $t$, then $G$ is non-synchronizing, and the kernel of $t$  induces a uniform partition. Therefore it is natural to ask the following question: is it true that a primitive group $G$ synchronizes every map  $t$ whose kernel is non-uniform?  We call  such groups {\em almost synchronizing}. Our  guess is that, with the possible exception of finitely many groups, every primitive group is almost synchronizing, but such a proof will likely require the (yet to be done) classification of the 1-transitive synchronizing groups. (It is obvious that 2-homogenous, or 2-set transitive, groups are synchronizing.)

We observe that an almost synchronizing group is primitive. If $G$ is imprimitive, then it fails to synchronize a map $t$ which collapses one block of imprimitivity into one single point in that block, and is the identity elsewhere. This map is not uniform. In addition $G$, together with $t$, cannot generate a constant since the rank  of any map in this semigroup will have at least the number of imprimitivity blocks. 

The goal of this paper is to provide two methods to find almost synchronizing groups. 
For the first method,  we introduce an hierarchy in the class of non-synchronizing groups and prove that the first level of that hierarchy, what we call {\em parameter $2$ non-synchronizing groups}, are almost synchronizing (provided they satisfy also a weak extra condition that holds for all parameter $2$ non-synchronizing groups we know); in addition we prove that there are infinitely many parameter $2$ non-synchronizing groups. We hope that this approach can be generalized for the parameter $k$ non-synchronizing groups (for $k>2$) and eventually a full classification of the almost non-synchronizing groups will be achieved. (We should observe, however, that in this paper we do not provide a classification of the parameter $2$ non-synchronizing groups; we only prove that there are infinite families of such groups.) 

In the second method we use a totally orthogonal approach, much more combinatorial, and based on some recent results by Godsil and Royle \cite{gr}. Again we prove that this method leads to infinite families of almost synchronizing groups, with little overlap with the groups found using the first method.

\section{First method: notation and preliminary results}

In this paper we make the global assumption that $G\le S_n$
is a primitive group.

Given a set $S\subseteq X$ with $1<|S|<n$,
we define a graph $\Gamma_S$ on the vertex set $X$
by the rule that $x\sim y$ for $x \ne y$ if and only if there is an element $g\in G$
with $\{x,y\}g\subseteq S$. Following the standard convention, the complement of $\gam_{S}$ will be denoted $\bgam_{S}$; in addition the neighborhood of an element $x$ in $\bgam_{S}$ will be denoted by $\bgam_{S}(x):=\{y\mid x\sim y\}$ and the closed neighborhood of $x$ in $\bgam_{S}$ will be denoted by $\bgam_{S}[x]:=\{x\}\cup \{y\mid x\sim y\}$.   Clearly the graph $\bgam_S$ is $G$-invariant
(that is, it is an edge-disjoint union of orbital graphs for $G$), and
hence is vertex-transitive; and $S$ is an independent set in $\bgam_S$.

\begin{lemma}
Suppose that $T$ is a clique in $\bgam_S$ which meets every $G$-translate
of $S$. Then
\[\bigcap_{x\in T}{\bgam}_S[x]=T.\]
\end{lemma}

\begin{proof}
The hypothesis that $T$ meets every $G$-translate of $S$ is equivalent to
saying that $|S|\cdot|T|=n$; so $T$ is a maximal clique of $\bgam_{S}$. Clearly $T$ is
contained in the intersection on the left; if the intersection contained an
additional point $z$, then $T\cup\{z\}$ would be a clique, a contradiction.
\end{proof}

Let $S$ and $T$ be as in the previous lemma.
We define the parameter $m(S,T)$ to be the smallest number $m$ such that
the intersection of the closed neighbourhoods of \emph{any} $m$ points of
$T$ in $\bgam_S$ is equal to $T$. Said otherwise, $m(S,T)$ is one more than
the maximum number of neighbours in $T$ of any point outside $T$ in the
graph $\bgam_S$. Clearly $m(S,T)\le|T|$. Moreover, 
$m(S,T)\ge2$. For suppose that $m(S,T)=1$. Then $T={\bgam}_S[x]$
for any $x\in S$; so $T$ is a connected component of $\bgam_S$, contradicting
the assumed primitivity of $G$. 

The following lemma is immediate from the definitions. 
 
\begin{lemma}
Suppose that $m(S,T)=2$. Then the $G$-invariant graph $\bgam_S$ has the
property that every point $x$ has a neighbourhood $\bgam_S(x)$ having a
connected component which is a complete graph. In particular, if $\bgam_S$ is
edge-transitive, then $\bgam_S(x)$ is a disjoint union of complete graphs.
\end{lemma}

In order to formulate our main theorem, we need some more definitions. 

A partition $P$ of $X$ is a \emph{non-synchronizing partition} for $G$ if
there is a subset $S$ of $X$ such that $|S\cap T|=1$ for every part $T$ of $P$.
The set $S$ is called a \emph{$G$-regular section} for $P$. It follows from earlier remarks that a group $G$ is non-synchronizing if and
only if it has a non-synchronizing partition.

We define a parameter $m(G)$ as follows: $m(G)$ is the maximum of $m(S,T)$,
over all pairs for which $T$ is a part of a non-synchronizing partition
$P$ for which $S$ is a $G$-regular section.
Note that such a set $T$ is indeed a clique in $\Gamma_S$ satisfying
$|S|\cdot|T|=n$. Note also that the existence of such $S,T,P$ is equivalent
to the group $G$ being non-synchronizing; the parameter $m(G)$ is not defined
unless $G$ is non-synchronizing. We call $m(G)$ the \emph{non-synchronizing
parameter} of $G$.

We also need an additional condition. For two distinct $G$-non-synchronizing
partitions $P$ and $P'$ of the same
rank $k$, let $M(P,P') = \frac{|P \cap P'|}{k}$; and set $M(G)$ to be
the maximum over all such pairs.

\section{First method: the main result}

We will show that every non-synchronizing group $G \leq S_n$  
with $m(G)=2$ and $M(G) \le 1/2  $ is almost synchronizing, that is, it 
synchronizes every non-uniform transformation. (Note that condition $M(G) \le 1/2  $ is satisfied by all examples of groups with $m(G)=2$ known to us.) 

Throughout this section let $G\leq S_{n}$ be such a group, and $t\in T_{n}$  be a map whose kernel is a non-uniform partition. 
Set 
$U=\langle G, t \rangle$, and let  $k_U$ be the smallest rank of the transformations in $U$, say $\rank(p)=k_U$, for some $p\in U$. It follows that  all the maps in $\langle G,p\rangle$ either are in $G$ or have rank $k_U$. Note that $k_U < \rank(t)$, as (by \cite{pmn}), for some $g\in G$, we have $\rank(tgt)<\rank(t)$. Our goal is to show that $k_U=1$.

\begin{lemma} \label{lem:exist q}
With notation as above, there exists a map $q\in U$ such that $\rank(q)>k_U$,
 $\rank(qgq)\in \{\rank(q),k_U\}$
 for all $g\in G$, and that there exists an $h \in G$ with $\rank(qhq)=k_U$.
\end{lemma}
\begin{proof}
Let $$M=\{a\in \langle G,t\rangle \mid \rank(a)>k_U \mbox{ and } (\exists h\in G)\ \rank(aha)=k_U\}.$$

As $\rank(t)>k_U$,
there exists elements $g_1, \ldots,g_n \in G$, such that $\rank(tg_{1}t\ldots t g_{n-1}t)>k_U$ and
$\rank(tg_{1}t\ldots t g_{n}t)=k_U$. This implies that 
$$\rank\left((tg_{1}t\ldots g_{n-1} t)g_{n}(tg_{1}t\ldots tg_{n-1}t)\right)\le  \rank(tg_{1}t\ldots g_{n}t)=k_U$$ 
and in fact
$\rank\left((tg_{1}t\ldots g_{n-1}t)g_{n}(tg_{1}t\ldots tg_{n-1}t)\right) =k_U$, as $k_U$ is the smallest rank possible.
Hence $M$ is non-empty.
Let $q\in M$ be of minimal rank. We claim that $q$ satisfies the desired property. 

By definition of $M$, $\rank(q)>k_U$ and there exist $h\in G$, with $\rank(qhq)=k_U$.  Let $g\in G$. If $\rank(qgq)>k_U$, then  $\rank((qgq)h(qgq))=\rank(qg(qhq)gq)\leq \rank(qhq)=k_U$. As above, it follows that  $\rank((qgq)h(qgq))=k_U$. 
 Thus $qgq\in M$ and as $q$ has minimal rank in $M$, $\rank(qgq)=\rank(q)$. The result follows.  
\end{proof}

Fix any map $q$ satisfying the condition of Lemma \ref{lem:exist q}, and denote by $Q_{q}\subseteq G$ the set of group elements $g$ such that $\rank(qgq)=k_U<\rank(q)$. (If $q$ is clear we might write $Q$ instead of $Q_{q}$.)

We will make repeated use of transformations of the form $qgq$ and $qhqgq$, where $h,g \in G$. 
Let the partition induced by $Ker( q )$,  $Ker( qgq )$ and $Ker( qhqgq )$ be denoted by $K$, $K_{g}$, and $K_{h,g}$, 
respectively, and let $X_{g}=Xqgq$,
$X_{h,g}=Xqhqgq$. We will occasionally blur the distinction between partition and equivalence relations.

Observe that, for all $(g,h,h')\in Q \times G\times Q$, we have that $X_{g}h$ is a section for $K_{h'}$. Otherwise, $qgqhqh'q$ would have rank smaller than $\rank(qgq)=k_U$. 
Therefore $\{X_{g}\mid g\in Q\}$ is a set of $G$-regular sections for every partition in the set $\{K_{g}\mid g\in Q\}$. 

The following Lemma is a direct consequence of $m(G)=2$.

\begin{lemma} \label{lem:2intersect}
Suppose that $S \subseteq X$ is a $G$-section for two (necessarily  non-synchronizing) partitions $P$ and $P'$. 
If $x \ne y$ and $y \in [x]_P \cap [x]_{P'}$,
then $[x]_P=[x]_{P'}$.
\end{lemma}
\begin{proof} As $m(G)=2$, we get that
$$ [x]_P=\bg_{S}[x]\cap \bg_{S}[y]=[x]_{P'}. $$ 
\end{proof}

The next two results will show that if  $G$ is not almost synchronizing, then the transformation $q$ will 
take a very specific form.

\begin{lemma} \label{lcol} 
If $k_U>1$, then for any $g \in G$ , a block of $K_g$ is either a block of 
$K$ or a union of singleton blocks from $K$. 
\end{lemma}
\begin{proof} We may assume that $g\in Q$, the result being trivially true otherwise.
Let $A_1$ and $A_2$  be distinct blocks of $K$ such that $A_1 \cup A_2$ is contained in a 
block $B$ of $K_g$, say $A_1 q =\{a_1\}$ and $A_2 q =\{a_2\}$ with $a_1 \ne a_2$. 

As $k_U>1$, $\{B,X-B\}$ is a partition of $X$. 
By primitivity, there exists an $h \in G$ such that $a_1 h \in B$ and $ a_2 h\not\in B$. Note that for any
$g' \in G$, 
$X_g{g'}$ is a section for both $K_g$ (as noted above) and $K_{h,g}$ (otherwise, $qgqg'qhqgq$ would have rank smaller then $qgq$).
So, by the previous lemma, we know that any two blocks of $K_g $ and $K_{h,g}$ that have two elements in common
must agree.

However, the $A_1$-block of $K_g$ contains $A_2$ while the $A_1$-block of $K_{h,g}$ does not. Hence $A_1$
cannot contain two elements and must be a singleton. The general result follows by symmetry.
\end{proof}


\begin{lemma} \label{ker partitions}
Let $k_U>1$.
Then the partition $K$ consists of $r$ sets of order $1$, and $s$ sets of order $p>1$,
 such that $r$ is a multiple of $p$, $ r <  sp$, $r \ge 1$, $s \ge 1$.
\end{lemma}
\begin{proof}
Pick any $g \in Q_q$. In order for $qgq$ to have smaller rank then $q$,
 only singleton classes of the kernel can form larger blocks of $K_g$ by the previous lemma. As $qgq$ is uniform by the theorem of Neumann~\cite{pmn}, the non-singleton classes must have the same size, say $p$.   Moreover, the singletons must combine in multiples of $p$. As $q$
is non-uniform we also get that $s,r \ge 1$

Now pick an element $ d \notin im(q)$ and an image $f$ of a non-singleton block of $K$. Let $h \in G$ be such that $dh$ is in the union of the singleton classes of
$K$ and $fh$ in its complement. Note that $\rank(qhq)< \rank(q)$ as $dhq \notin X_h$, and hence 
$\rank(qhq)=k_U$. However this requires that $h$ maps all images of singleton classes into non-singleton classes (of $K$). As $h$ also maps $f$ into this set, it follows that $r$ is strictly less than $sp$.
\end{proof}

The lemma in particular shows that for any $g \in Q$, over half of the $K_g$-blocks are also blocks of $K$, as opposed to 
blocks made up of singleton classes of $K$.
We are now ready to prove our main theorem.

\begin{theorem}\label{main}

Let $G \le S_n$ be a non-synchronizing group with $m(G)=2$ and $M(G) \le 1/2$.  If $t \in T_n$ is a transformation such that $Ker(t)$ is a non-uniform 
partition, then $\langle G,t\rangle$ contains a constant function; that is, $G$ is almost synchronizing.
\end{theorem}

\begin{proof}
Assume otherwise, i.e.\ that in the notation from above 
$k_U>1$, and let $q$ be an transformation satisfying the conditions of
Lemma \ref{lem:exist q} and  Lemma \ref{ker partitions}.
Let $g \in Q_q$ such that  $\rank(qgq)=k_U$, and let $(a,b) \in K_g$ such that $(a,b) \notin K$.
Let $h \in G$ satisfy $aqh \in B$,  $ bqh \notin B$ for some block $B$ of $K_g$. 

Consider the functions $qhqgq$ and $qgq$, both of rank $k_U$. Note that $X_g$ must be a $G$-regular section for both 
$K_{h,g}$ and $K_g$. All non-singleton blocks of $K$ are in $K_g$ by Lemma \ref{lcol}. The same holds for $K_{h,g}$, as it  needs to be unifom and contained in $K$.

By Lemma  \ref{ker partitions} we have that the number of these blocks  is larger than  $k_U/2$. Our assumption $M(G) \le 1/2$
now shows that
$K_g=K_{h,g}$. However, it is easy to check that $(a,b) \notin K_{h,g}$. 

Hence $k_U =1$ and  $\langle G,t\rangle$ contains a constant function.
\end{proof}

\section{First method: examples}

We now give two infinite classes of examples and one sporadic example of
groups which satisfy the hypotheses of the main theorem, and hence which
synchronize any non-uniform mapping.

\paragraph{A sporadic example} The automorphism group of the $3\times3$ grid
(the wreath product of $S_3$ with $S_2$) and its primitive subgroup of
index~$2$, are examples. For there are four non-synchronizing partitions:
the rows and columns of the grid, and two partitions into diagonal sets 
(corresponding to the even and odd permutations of $\{1,2,3\}$ respectively).
It is easy to see that, for these groups, $m(G)=2$, and $M(G)=0$ (since
distinct non-synchronizing partitions have no sets in common).

\paragraph{Examples from projective planes} Let $G$ be the group
$\mathrm{P}\Gamma\mathrm{L}(3,q)\cdot2$ or a subgroup containing
$\mathrm{PSL}(3,q)\cdot2$, where $2$ denotes the \emph{inverse transpose
automorphism}. The subgroup of index~$2$ acts on the projective plane over
the field $\mathrm{GF}(q)$, and the outer automorphism induces a duality of
the plane interchanging points and lines. Then $G$ acts on $X$, the set of
\emph{flags} (incident point-line pairs) in the projective plane. There
are three orbits of $G$ on pairs of flags:
\begin{itemize}
\item $O_1=\{\{(P,L),(P',L')\}: P=P'\hbox{ or }L=L'\}$;
\item $O_2=\{(P,L),(P',L')\}: P\in L'\hbox{ or }P'\in L\}$;
\item $O_3$, the remaining pairs (consisting of ``opposite'' flags).
\end{itemize}
There are only two non-synchronizing partitions for $G$. In the first
partition, two flags are in the same part if they share a point; in the second,
they are in the same part if they share a line. These two partitions have no
common parts, so $M(G)=0$.

A $G$-regular section for either partition consists of a set of $q^2+q+1$
flags, no two sharing a point or a line. Such a set can be constructed as
an orbit on flags of a \emph{Singer cycle} of $G$, a cyclic subgroup of order
$q^2+q+1$ permuting points and lines transitively. It must contain a pair
of flags lying in $O_2$ and a pair lying in $O_3$. So the edges of 
$\bar\Gamma_S$ are the pairs in $O_1$. Now any two flags adjacent in this graph
determine a common point or line, so the intersection of their closed 
neighbourhood consists of all flags using this point or line. Thus
$m(G)=2$.

\paragraph{Examples from symplectic groups} Let $G$ be the projective 
symplectic group $\mathrm{PSp}(4,q)$, where $q$ is a power of~$2$, or any
group obtained by adjoining field automorphisms. The group $G$ acts on the
symplectic generalized quadrangle $W(q)$ whose points are those of the
$3$-dimensional projective space over $\mathrm{GF}(q)$, and whose lines are
the lines of projective space which are totally isotropic with respect to
a symplectic polarity. The only $G$-invariant graphs are the
\emph{collinearity graph} of the quadrangle (two vertices joined if they are
orthogonal with respect to the symplectic form) and its complement.

The only non-synchronizing partitions for $G$ are \emph{spreads} of lines,
families of pairwise disjoint lines covering the point set; the $G$-section
for such a partition is an \emph{ovoid}, a set of points meeting every line
in precisely one point. (The other possibility for a non-synchronizing 
partition would be a partition into ovoids; but a theorem of
Butler~\cite{Butler1} shows that any two ovoids intersect in an odd number
of points, so no such partition exists.) A spread consists of $q^2+1$ lines,
each of cardinality m$q+1$. If $S$ is an ovoid, then $\bar\Gamma_S$ is the
collinearity graph. A generalized quadrangle has the property that any point
$x$ not on a line $l$ is collinear with a unique point of $l$; so $m(G)=2$.

Since the characteristic is $2$, the generalized quadrangle admits a 
\emph{duality} interchanging points and lines and preserving incidence.
Such a duality interchanges spreads with ovoids; so the maximum size of the
intersection of two spreads is equal to the maximum size of the intersection
of two ovoids. This parameter was first studied by Glynn~\cite{Glynn}, who
showed that two ovoids intersect in at most $q(q-1)/2$ points. (A more
accessible proof is in Butler~\cite{Butler2}.) So
\[M(G)\le\frac{q(q-1)}{2(q^2+1)}<\frac{1}{2},\]
and our main theorem shows that $G$ is almost synchronizing.

\section{Second method: graphs}

The core of a graph $\Gamma$ is the smallest graph $\Delta$ that is homomorphically equivalent to $\Gamma$ (that is, there exist homomorphisms in both directions). The core of $\Gamma$ is unique up to isomorphism and is an induced
subgraph of $\Gamma$. It is well known that all the endomorphisms of a core are automorphisms. 

We will denote the clique number of a graph $\Gamma$ by $\omega(\Gamma)$, and its chromatic number by $\chi (\Gamma)$.

There is a natural connection between transformation monoids and graphs.
To every graph $X$, we associate its \emph{endomorphism monoid} $\End(X)$.
In the other direction, given a transformation monoid $M$ on $\Omega$, we
define a graph $\Gr(M)$ on $\Omega$ by the rule that $x\sim y$ in $\Gr(X)$
if and only if there is no element $f\in M$ satisfying $xf=yf$. The next
result gives some simple properties of this construction. This is an
extension of the methods in \cite{ck}.

\begin{theorem}
\begin{itemize}
\item[(a)] $\Gr(M)$ has complete core; that is, its clique number and chromatic
number are equal.
\item[(b)] $M\le\End(\Gr(M))$.
\end{itemize}
\end{theorem}

\paragraph{Proof} (a) Let $f$ be an element of minimal rank in $M$. Then
the induced subgraph on the image of $f$ is complete; for if $x,y$ are not
joined, then there exists $g\in M$ with $xg=yg$, so $fg$ has smaller rank
than $f$. Now the map $f$ is a colouring of $\Gr(M)$ (since if two vertices
have the same image under $f$ they cannot be adjacent), and the number of
colours is equal to the rank of $f$.

(b) Take $f\in M$, and suppose that $f$ is not an endomorphism of $\Gr(M)$.
Now $f$ cannot collapse an edge of $\Gr(M)$ to a single vertex, by definition;
so it must map an edge $xy$ to a non-edge $uv$. But then there is $g\in M$
with $ug=vg$; so $x(fg)=y(fg)$, a contradiction.

\begin{theorem}
Let $f$ be a map not synchronized by the permutation group $G$. Then there is
a $G$-invariant graph $X$ with $f\in\End(X)$ and $\omega(X)=\chi(X)$.
\end{theorem}

\begin{proof} Let $M=\langle G,f\rangle$ and $X=\Gr(M)$.
\end{proof}

We conclude that a group $G$ is non-synchronizing if and only if there is a $G$-invariant graph $X$, not complete or null, with $\omega(X)=\chi(X)$. 

\section{Second method: results of Godsil and Royle}

Godsil and Royle~\cite{gr} have shown that certain strongly regular graphs $X$
are what they call \emph{pseudo cores}: this means that every endomorphism of
$X$ is either an automorphism or a colouring. (We recall that a  core is a graph for
which every endomorphism is an automorphism.)

If we can show that a primitive permutation group $G$ has the property that
the only $G$-invariant graphs $X$ with $\omega(X)=\chi(X)$ are pseudocores,
then we have shown that $G$ is almost synchronizing; for if $f$ is not
synchronized by $G$, then $f$ is a colouring of a $G$-invariant graph $X$;
so $f$ has minimal rank in $\End(X)$, and is uniform.

In each case, a necessary condition is that any clique of maximum size in
the graph is a line of the corresponding geometry. There is an inequality
which guarantees this, which we have placed in square brackets. If this
condition could be shown directly, this inequality would not be required.
No condition is required for generalized quadrangles, since an edge lies
in a unique maximal clique, namely a line.

Among the graphs that Godsil and Royle show to be pseudo cores are
\begin{itemize}
\item[(a)] the line graphs of $2$-$(v,k,1)$ designs with $k>2$
[and $v>k(k^2-2k+2)$].
\item[(b)] the graphs on $n^2$ points obtained from orthogonal arrays
$O(2,k,n)$, where $k>2$ [and $n>(k-1)^2$].
\item[(c)] The collinearity graphs of generalized quadrangles with $s,t>1$.
\end{itemize}
We refer to their paper for definitions.

\section{Second method: Examples}

\begin{theorem}
\begin{enumerate}
\item[(a)] Let $G$ be a subgroup of $\PGamL(n,q)$ containing $\PSL(n,q)$, where
$n\ge5$, acting on the lines of the projective space. Then $G$ is almost
synchronizing.
\item[(b)] Let $G$ be the semidirect product of the additive group of $\GF(p^2)$
by the subgroup of index~$2$ in the multiplicative group of the field, where
$p$ is prime. Then $G$ is almost synchronizing.
\item[(c)] Let $G$ be the symplectic group $\PSp(4,q)$ or be obtained from it
by adjoining field automorphisms, where $q$ is a power of $2$. Then $G$ is
almost synchronizing.
\end{enumerate}
\end{theorem}

\paragraph{Proof} (a) Here $G$ is a rank~$3$ group, and so there are just two
non-trivial $G$-invariant graphs, the concurrence graph of the lines of the
projective space and its complement.

The lines form a $2$-$(v,q+1,1)$ design, where $v=(q^n-1)/(q-1)$. We see that
$v>k(k^2-2k+2)$ if and only if $n\ge5$. So in this case, the graph is a
pseudocore. (The cliques of maximum size consist of all lines through a point.
If $n=4$, there are further such cliques, namely all the lines in a plane, and
the method fails.)

In the complementary graph, the clique number is at most $(q^n-1)/(q^2-1)$,
since clearly we cannot find more than this number of disjoint lines. However,
a maximal coclique has size $(q^{n-1}-1)/(q-1)$ (these are cliques in the
complementary graph, and consist of all lines through a point), and any two
of these cocliques intersect; so the chromatic number is strictly greater
than
\[\frac{(q^n-1)(q^{n-1}-1)/(q^2-1)(q-1)}{(q^{n-1}-1)/(q-1)},\]
the numerator being the total number of vertices; so the clique number and
chromatic number are not equal.

Note that the geometry has a \emph{parallelism} (a partition of the lines
into spreads)  only if $n$ is even; so for $n$ odd, the groups are
synchronizing. Parallelisms have been shown to exist for $n=4$ \cite{beutel} and $q=2$ \cite{baker}, and are conjectured to exist for all even $n$.

\medskip

(b) Again the group has rank~$3$, and the two $G$-invariant graphs correspond
to orthogonal arrays $(2,(p+1)/2,p)$. The vertices in each case are points of
the affine space, and the $p+1$ directions of lines are partitioned into two
subsets of $(p+1)/2$ such that in each graph, two vertices are joined if the
line joining them has direction lying in the corresponding set.
In this case, the inequality given in square brackets is not satisfied; but
the graphs are pseudo cores provided that we can prove otherwise that there
are no cliques of size $p$ other than lines of the affine space.

This follows from Theorem 24' of R\'edei~\cite{redei}, according to which a set
of $p$ points in the affine plane which is not a line determines at least
$(p+3)/2$ directions. Since only $(p+1)/2$ directions correspond to
adjacency, no such set is a clique.

So both graphs are pseudo cores, and the result follows.

\medskip

(c) For any classical generalized quadrangle, the automorphism group is a
rank~$3$ group whose invariant graphs are the collinearity graph and its
complement. The number of points is $(s+1)(st+1)$, and a line has $s+1$
points.

The collinearity graph has clique number $s+1$ (the cliques are lines);
an independent set meet each line in at most one point, and so has size at
most $st+1$, with equality if and only it is an ovoid. So the chromatic
number is $s+1$ if and only if there is a partition into ovoids.

In the complement of the collinearity graph, as we have seen, the clique
number is at most $st+1$, with equality if and only if there is an ovoid; the
chromatic number is at least $st+1$, with equality if and only if there is a
spread (a set of lines partitioning the point set).

So one of these groups is almost synchronizing if the generalized quadrangle
has ovoids and spreads but no partition into ovoids.

The symplectic generalized quadrangles in even characteristic have these
properties. Note that these examples are  also 
covered by our first method.

\section{Some more examples}

Here is a class of examples that use the same technique as that of Godsil and
Royle but are not covered by their results. The symmetric group $S_m$
acting on $2$-sets (for $m\ge5$) is synchronizing if and only if $m$ is odd.
It is the automorphism group of a $2$-$(m,2,1)$ design; but this is not
covered since Godsil and Royle assume that the block size is greater than $2$.

\begin{theorem}
The symmetric group $S_m$ acting on $2$-sets is almost synchronizing if $m$
is even and $m\ge6$.
\end{theorem}

\paragraph{Proof}
As usual the groups have rank~$3$, and the graphs we have to consider are the
line graph of the complete graph $K_m$ and its complement.

The line graph of $K_m$ has clique number $m-1$ (take all the edges containing
a given point) and chromatic number $m-1$ if $m$ is even (take a
$1$-factorization of $K_m$). The complement has clique number $m/2$ and
chromatic number strictly greater. (This can be seen by an argument similar to
the one we used for the projective space: the maximal cliques have size $m-1$
and any two intersect. In fact a result of Lov\'asz~\cite{lovasz} shows that
the chromatic number is $m-2$.) So we only need consider the line graph.

Suppose that $f$ is an endomorphism of $L(K_m)$ which is not an
automorphism. Let $C_i$ denote the $(m-1)$-clique consisting of all edges
through the point $i$. Each such clique must be mapped to another by $f$;
we have to show that they all collapse to a single clique.

We will use the fact that the octahedron (the induced subgraph on the six
$2$-subsets of a $4$-set) is a pseudo-core, that is, every endomorphism which
is not an automorphism maps it onto a triangle. This is easy to prove
directly. Below we will say ``the octahedron $abcd$'' to mean the octahedron
formed by the six pairs from this set.

Suppose that $12$ and $34$ map to the same point, which we may suppose to
be $12$. Then $C_1$ maps to either $C_1$ or $C_2$; without loss of generality,
$C_1$ is mapped to $C_1$.

Consider the octahedron $1234$. Since $12$ and $34$ map to
$12$, and $13$ and $14$ both map into $C_1$, the whole
set maps into $C_1$.

Next we claim that $C_2$ maps to $C_1$. For it contains three points $12$,
$23$, $24$ which all map into $C_1$.

Finally, every pair $ij$ maps into $C_1$. For consider the octahedron
$12ij$; we know that all except $ij$ map into $C_1$, so $ij$ does as
well.

This method also applies to the alternating group $A_{n}$ and the Mathieu groups $M_{12}$ and $M_{24}$. 

\section{Complexity}

The person inside the dungeon wants to find as quickly as possible if the
permutations in the automaton generate an almost synchronizing group. So
that person needs a fast algorithm to decide the question in terms of the
given group generators.

There exist efficient polynomial-time algorithms for deciding whether a
permutation group with a given set of generators is transitive, or primitive,
or $2$-transitive. We refer to Seress~\cite{seress} for a survey of these
algorithms. Indeed, if we are given an unnecessarily large set of generators,
it can be transformed with \emph{polynomial delay} into a set of size at
most $n-1$ (this means that a polynomial amount of computation is done after
reading each generator). After this has been done, the tests for transitivity
etc.\ are polynomial in $n$.

The synchronizing property lies between primitivity and $2$-transitivity, and
currently no efficient algorithm is known. In fact, the best known algorithm
is based on the considerations in Section 5. It consists of the following
steps:
\begin{itemize}
\item[(a)] construct all the non-trivial $G$-invariant graphs;
\item[(b)] for each such graph $X$, decide whether $\omega(X)=\chi(X)$.
\end{itemize}
In the first step, the number of graphs is $2^r-2$, where $r$ is the number of
$G$-orbits on the set of $2$-subsets. Although this is exponential, for many
primitive groups the number $r$ is relatively small. The second step, however,
is NP-hard for general graphs, though the possibility that it is easier for
graphs with primitive automorphism group remains open.

We do not currently have any reasonable algorithm for testing whether a group
is almost synchronizing. However if, as we suspect, this property is equivalent
to primitivity, there would be a polynomial-time test!

For the particular groups in our examples, recognition by standard
group-theoretic algorithms is easy and can certainly be done in random
polynomial time.

\section{Problems}

The results in this paper prompt a number of very natural problems that might attract the attention of experts in computer science, combinatorics and geometry, groups and semigroups, linear algebra and matrix theory. 

We could find no group $G$ such that $m(G)=2$ and $M(G)>\frac{1}{2}$ and hence we do not know if the latter condition is redundant. 

\begin{prob}
Is it true that $m(G)=2$ implies $M(G)\leq \frac{1}{2}$?
\end{prob}

Even if this is not true, what really matters regarding the classification of almost non-synchrozoning groups, is to prove a result similar to  Theorem \ref{main}, but with no assumptions on $M(G)$. 

\begin{prob}
Is it possible to prove Theorem \ref{main} without any assumption about $M(G)$?
\end{prob}

Since we have a general result linking almost non-synchronizing groups and  parameter 2  non-synchronizing groups, the next goal is a classification of these groups. 

\begin{prob}
Classify the primitive non-synchronizing groups $G$ such that $m(G)=2$.   
\end{prob}

The parameter $m(G)$ induces an hierarchy on the class of non-synchronizing groups. We hope this hierarchy helps splitting the classification of synchronizing groups and almost synchronizing groups down to more tractable subclasses. 

\begin{prob}
For each $k\geq 2$, classify the primitive almost synchronizing groups such that $m(G)=k$. 
\end{prob}

The main conjecture in this paper is that primitive groups are almost synchronizing. 

\begin{prob}
Is it true that primitive groups are almost synchronizing?
\end{prob}

In this setting, and regarding algorithms,   the main problem is the following. 

\begin{prob}
Find an efficient algorithm to decide if a given group is synchronizing. 
\end{prob}

\paragraph{Acknowledgment}
We are grateful to Simeon Ball and Tim Penttila for help with the literature
on ovoids and spreads in symplectic generalized quadrangles.

The first author was partially supported  by FCT through the following projects: Strategic Project of Centro de \'Algebra da Universidade de Lisboa (PEst-OE/MAT/UI1043/2011); and Project Computations in groups and semigroups (PTDC/MAT/101993/2008).

The e second author  has received funding from the
European Union Seventh Framework Programme (FP7/2007-2013) under
grant agreement no.\ PCOFUND-GA-2009-246542 and from the Foundation for 
Science and Technology of Portugal.

The third author is grateful to the Center of Algebra of the University of Lisbon 
for supporting a visit to the Centre in which some of this research was done.

\end{document}